\documentclass[11pt]{amsart}
\usepackage{amsmath}
\usepackage{amssymb}
\usepackage{amsfonts}
\usepackage{amsthm}
\usepackage{mathrsfs}
\usepackage{comment}
\usepackage[normalem]{ulem}
\usepackage{enumerate}
\usepackage[
textwidth=405pt,
textheight=615pt,
hmarginratio=1:1,
vmarginratio=1:1]{geometry}
\usepackage{mathtools}
\mathtoolsset{showonlyrefs}
\usepackage{hyperref}
\hypersetup{colorlinks=true,
linktoc=all,linkcolor=black,
citecolor=black}

\newcommand{\id}{\operatorname{\text{\bf I}}}

\newcommand{\calH}{\mathscr H}
\newcommand{\C}{\mathbb{C}}
\newcommand{\D}{\mathbb{D}}
\newcommand{\T}{\mathbb{T}}
\newcommand{\Te}{\mathbb{T}}

\newcommand{\Z}{\mathbb{Z}}
\newcommand{\R}{\mathbb{R}}

\newcommand{\calE}{\mathcal{E}}
\newcommand{\Fspace}{\mathscr{F}(\C)}
\newcommand{\Fshort}{\mathscr{F}}
\newcommand{\diffA}{\mathrm{dA}}

\newcommand{\diff}{\mathrm{d}}

\newcommand{\Top}{\mathbf{T}}

\newcommand{\Mop}{\mathbf{M}}

\newcommand{\Uop}{\mathbf{U}}
\newcommand{\Bop}{\mathbf{B}}
\newcommand{\Fop}{\mathbf{F}}

\newcommand{\e}{\mathrm{e}}

\newcommand{\imag}{\mathrm{i}}

\newcommand{\im}{\mathrm{Im}\,}
\renewcommand{\hat}{\widehat}

\newtheorem{thm}{Theorem}[section]

\newtheorem{lem}[thm]{Lemma}

\newtheorem{prop}[thm]{Proposition}

\theoremstyle{definition}
\newtheorem{defn}[thm]{Definition}

\newtheorem{prob}[thm]{Problem}

\theoremstyle{remark}
\newtheorem{rem}[thm]{Remark}

\usepackage{graphicx}

\numberwithin{equation}{section}
\title[Deep zero problems]
{Deep zero problems}
\makeatletter
\let\@wraptoccontribs\wraptoccontribs
\makeatother

\begin{document}
\author[Hedenmalm]{Haakan Hedenmalm}

\address{Department of Mathematics
\\
The Royal Institute of Technology
\\
S -- 100 44 Stockholm
\\
SWEDEN
\\ 
$\&$ Department  of Mathematics and Computer Sciences
\\
St. Petersburg State University
\\
St. Petersburg
\\
RUSSIA
\\
$\&$ Department of Mathematics and Statistics
\\
University of Reading
\\
Reading
\\
U.K.}
\email{haakanh@kth.se}

\keywords{Bergman kernel, Siegel-Bargmann space, Bargmann-Fock space,
uniqueness set, interpolation, sampling}

\subjclass{Primary 46E20, 30E05, 46E22}

\date{\today}

\begin{abstract} 
We introduce a novel collection of uniqueness problems, with related 
interpolation and sampling issues. We call them \emph{deep zero}
problems, as they are concerned with local properties at a small number
of given points.

\end{abstract}

\maketitle


\section{Introduction}

We consider an infinite-dimensional Hilbert space $\calH$ of holomorphic 
functions on a domain $\Omega$ in the complex plane $\C$. The inner product 
on $\calH$ is such that the point evaluation $f\mapsto f(\lambda)$ is continuous
$\calH\to\C$ for each $\lambda\in\C$. We will consider a finite collection of points
$z_1,\ldots,z_N\in\Omega$, and place vanishing conditions on an infinite collection 
of the higher derivatives $f^{(j)}(z_k)$, for $k=1,\ldots,N$. We allow ourselves to replace
the vanishing condition by the vanishing of a sequence of finite linear combinations
of such expressions corresponding to each given point $z_k$. To not get bogged down by
a technical-looking definition in full generality, we consider a specific situation.
The \emph{Bargmann-Fock space} (or \emph{Siegel-Bargmann} space) $\Fspace$ is 
the Hilbert space of all entire functions
$f$ subject to the norm boundedness condition
\[
\|f\|^2_{\Fspace}:=\int_\C|f(z)|^2\diff\mu_G(z)<+\infty,
\]
where $\mu_G$ is the Gaussian probability measure
\[
\diff\mu_G(z):=\e^{-|z|^2}\diffA(z),\qquad\text{with}\quad \diffA(z):=\pi^ {-1}\diff x\diff y,
\]
under the usual convention $z=x+\imag y$. In terms of the Taylor expansion
\[
f(z)=\sum_{j=0}^{+\infty}\hat f(j)\,z^j,
\]
the norm finds a convenient expression in terms of the Taylor coefficients:
\begin{equation}
\|f\|_{\Fspace}^2=\sum_{j=0}^{+\infty}j!{|\hat f(j)|^2}.
\label{eq:normFspace}
\end{equation}
Correspondingly, the inner product on $\Fspace$ may be expressed both in terms of an
integral and as a sum involving Taylor coefficients:
\[
\langle f,g\rangle_{\Fspace}:=\int_\C f\bar g\,\diff\mu_G=
\sum_{j=0}^{+\infty}j!\hat f(j)\overline{\hat g(j)}.
\]
For background material on the Fock-Bargmann space, we refer to \cite{Zhu}, \cite{Folland}.

The point evaluations $f\mapsto f(z)$ are bounded linear functionals $\calH\to\C$, and
consequently, by the Riesz representation theorem, given by the inner product with an
element of $\calH$. In terms of the \emph{reproducing kernel } $K_\calH$ of $\Fspace$,
given by the formula
\[
K_\Fshort(z,w):=\e^{z\bar w},
\]
we have 
\begin{equation}
f(z)=\langle f,K_\Fshort(\cdot,z)\rangle_{\Fspace}=\int_\C f(w)K_\calH(z,w)\diff\mu_G(w)
=\int_\C \e^{z\bar w}\,f(w)\diff\mu_G(w).
\label{eq:reproprop1}
\end{equation}
Here, it is immediate that $K_\Fshort(\cdot,z)\in\Fspace$ for each $z\in\C$, so that we have found 
the element guaranteed by the Riesz representation theorem.
For a given $f\in\Fspace$, its $\alpha$-translate
\[
\Top_\alpha f(z):=f(z-\alpha),\qquad a,z\in\C,
\]
need not be in $\Fspace$ unless $\alpha=0$. There is a way to rectify this, by introducing the
\emph{Fock translates} 
\[
\Uop_\alpha f(z):=\e^{-\frac12|\alpha|^2+\bar \alpha z}f(z-\alpha), \qquad \alpha,z\in\C.
\]
Given that we may associate with an element of the Fock space the corresponding density
\[
|f(z)|^2\diff\mu_G(z)=|f(z)|^2\e^{-|z|^2}\diffA(z),
\]
whose integral equals the norm of $f\in\Fspace$, the interpretation of the Fock translate becomes
clearer as
\begin{multline}
|\Uop_\alpha f(z)|^2\diff\mu_G(z)=|\Uop_\alpha f(z)|^2\e^{-|z|^2}\diffA(z)
\\
=|f(z-\alpha)|^2\e^{-|z-\alpha|^2}
\diffA(z)=|f(z-\alpha)|^2\diff\mu_G(z-\alpha)
\label{eq:Focktransmod}
\end{multline}
means that we just perform an translation by $\alpha$ to the density. 
Consequently, the Fock translation acts isometrically on $\Fspace$,
\[
\|\Uop_\alpha f\|_{\Fspace}=\|f\|_{\Fspace},\qquad f\in\Fspace,
\]
and, moreover, if $g=\Uop_\alpha f$, we may solve for $f$: $f=\Uop_\alpha^{-1}g=\Uop_{-\alpha}g$. 
This means that the isometry $\Uop_a:\Fspace\to\Fspace$ is surjective, and hence unitary, meaning 
that its adjoint $\Uop_\alpha^\star$ meets $\Uop_\alpha^\star=\Uop_\alpha^{-1}$. 
More generally, the family of unitaries $\Uop_\alpha$ have the commutation property
\begin{equation}
\Uop_\alpha\Uop_\beta f=\e^{-\imag\im \alpha\bar\beta}\Uop_{\alpha+\beta}f=
\e^{-2\imag\im \alpha\bar\beta}\Uop_{\beta}\Uop_{\alpha}f,\qquad f\in\Fspace,\,\,\,\alpha,\beta\in\C,
\label{eq:commute1}
\end{equation}
which we may interpret to say that we have a projective unitary representation of the additive group
$\C$.
Let us agree to write $\mathbb{N}_0:=\{0,1,2,\ldots\}$. We consider the following toy problem.

\begin{prob}
("toy problem") Let $\calE$ be a collection of nonnegative integers, and 
$\calE^c:=\mathbb{N}_0\setminus\calE$ its complement. Suppose that $f\in\Fspace$ has the
property that
\[
f^{(j)}(0)=0,\qquad j\in\calE,
\]
while for some other point $\beta\in\C$ it holds that
\[
(\Uop_\beta f)^{(j)}(0)=0, \qquad j\in\calE^c.
\]
Does it then follow that $f=0$?
\end{prob}

Naturally, if $\beta=0$, then the two conditions assert that all the Taylor coefficients of $f$ vanish at
the origin, and then $f=0$ holds automatically. But what if $\beta\ne0$? An immediate observation is that
the condition of asking some derivatives to vanish at a given point is rather fragile in the sense that if we
multiply the function by some polynomial $p$, the condition for $f$ will generally not survive for the product $pf$. 
Another immediate observation is that we may consider, more generally, two subsets $\calE$ and $\calE'$ 
instead of just $\calE$ and its complement $\calE^c$.

We may also formulate natural interpolation and sampling problems. Note that in view of 
\eqref{eq:normFspace} and the equality $f^{(j)}(0)=j!\hat f(j)$, the norm on $\Fspace$ may be written in the form
\[
\|f\|_{\Fspace}^2=\sum_{j=0}^{+\infty}\frac{|f^{(j)}(0)|^2}{j!}.
\]

\begin{defn}
("toy deep interpolation") The \emph{interpolation question} asks whether
we can solve, with $f\in\Fspace$,  
\[
f^{(j)}(0)=a_j,\qquad j\in\calE,
\]
and 
\[
(\Uop_\beta f)^{(j)}(0)=b_j,\qquad j\in\calE^c,
\]
whenever 
\[
\sum_{j\in\calE}\frac{|a_j|^2}{j!}+\sum_{j\in\calE^c}\frac{|b_j|^2}{j!}<+\infty.
\]
\label{def:1}
\end{defn}

\begin{defn}
("toy deep sampling") The \emph{sampling question} asks whether
\[
\sum_{j\in\calE}\frac{|f^{(j)}(0)|^2}{j!}+\sum_{j\in\calE^c}
\frac{|(\Uop_\beta f)^{(j)}(0)|^2}{j!}\ge\epsilon
\|f\|^2_{\Fspace}
\]
holds for some positive constant $\epsilon$ and all $f\in\Fspace$.
\label{def:2}
\end{defn}

A weaker requirement than the deep sampling problem runs as follows. 
Is it true, at least, that  
there exists a positive constant $C(w)$ such that for each $f\in\Fspace$,
\begin{equation}
\label{eq:pteval}
|f(w)|\le C(w)\bigg\{\sum_{j\in\calE}\frac{|f^{(j)}(0)|^2}{j!}+
\sum_{j\in\calE^c}\frac{|(\Uop_\beta f)^{(j)}(0)|^2}{j!}\bigg\},\qquad w\in\C?
\end{equation}
A property analogous to \eqref{eq:pteval} was important in \cite{HRS} for a certain
constructive approach to the invariant suspaces of index $2$ in Bergman spaces.
We have the following theorem.

\begin{thm}
If $\calE=\{0,2,4,\ldots\}$ are the nonnegative even integers, and if $f\in\Fspace$ meets
\[
f^{(j)}(0)=0,\qquad j\in\calE,
\]
while for some point $\beta\in\C$ it holds that
\[
(\Uop_\beta f)^{(j)}(0)=0, \qquad j\in\calE^c,
\]
then $f=0$ identically. The same conclusion holds if $\calE$ consists of the nonnegative odd
integers as well. 
\label{thm:evenodd}
\end{thm}

After establishing this uniqueness property, if is reasonable to pass to the interpolation and
sampling problems. General reasoning would suggest that interpolation is impossible, since
the interpolation property often is the same as being a subset of a zero set in some uniform way.
In fact, we have the following.

\begin{thm}
Let $\calE$ denote either the nonnegative even integers, or the positive odd integers.
If $\beta\in\C$ with $\beta\ne0$, then we do not have a deep interpolation set in the sense
of Definition \ref{def:1}, nor do we have a deep sampling set in the sense of Definition \ref{def:2}.
\label{thm:samplint}
\end{thm}

Less precisely, the given deep uniqueness set is borderline, as it is neither sampling nor
interpolating. 

\begin{rem}
The unitary transformations $\Uop_\alpha$ are modelled on the translations $z\mapsto z-\alpha$.
Rotations $z\mapsto\rho z$, with $|\rho|=1$, also give rise to natural unitary transformations on
$\Fspace$. We combine them to the rigid motions group of elements $(\rho,\alpha)$ with associated 
map $z\mapsto \rho z-\alpha$, where $|\rho|=1$ and $\alpha\in\C$ is free. 
The group operation is $(\rho',\alpha')(\rho,\alpha)=(\rho'\rho,\rho'\alpha+\alpha')$. We put
\begin{equation}
\Uop_{(\rho,\alpha)}f(z):=\e^{-\frac12|\alpha|^2+\bar\alpha\rho z}f(\rho z-\alpha),
\label{eq:rigidmotions}
\end{equation}
which acts unitarily on $\Fspace$ and coincides with $\Uop_\alpha$ when $\rho=1$. The
commutation rule reads
\begin{equation}
\Uop_{(\rho,\alpha)}\Uop_{(\rho',\alpha')}=\e^{-\imag \im(\alpha\rho'\bar\alpha')}
\Uop_{(\rho'\rho,\rho'\alpha+\alpha')}.
\label{eq:commrule}
\end{equation}
\end{rem}

\section{The uniqueness for the even and odd integers}

We begin with a basic lemma, which is known. The proof is short and presented for completeness.

\begin{lem}
\label{lem:basic}
Suppose $f\in\Fspace$. Then
\[
|f(z)|\le \|f\|_{\Fspace}\sqrt{K_\Fshort(z,z)}=\|f\|_{\Fspace}\e^{\frac12|z|^2},\qquad z\in\C.
\]
This estimate is sharp, the extremals being the constant multiples of $K_\Fshort(\cdot,z)$.
On the other hand, it cannot be sharp at infinity, since
\[
\lim_{|z|\to+\infty}\e^{-\frac12|z|^2}|f(z)|\to0.
\]
\end{lem}

\begin{proof}
In view of the reproducing property \eqref{eq:reproprop1} and the Cauchy-Schwarz inequality, 
\[
|f(z)|\le\|f\|_{\Fspace}\|K_\Fshort(\cdot,z)\|_{\Fspace}
\]
with equality precisely when $f$ equals a constant multiple of $K_\Fshort(\cdot,z)$.
Moreover, since
\[
\|K_\Fshort(\cdot,z)\|_{\Fspace}^2=\langle K_\Fshort(\cdot,z),K_\Fshort(\cdot,z)\rangle_{\Fspace}
=K_\Fshort(z,z)=\e^{|z|^ 2},
\]
where in the last step we invoke the reproducing property \eqref{eq:reproprop1}, the asserted 
pointwise estimate is immediate. As for the last property, we observe that for polynomials $f$,
\[
\lim_{|z|\to+\infty}\e^{-\frac12|z|^2}|f(z)|\to0
\]
holds. Moreover, since the polynomials are dense in $\Fspace$, this property carries over to
all $f\in\Fspace$, by the same argument which gives that $c_0$ is norm closed in $\ell^\infty$. 
\end{proof}

The next result we need concerns the point spectrum of the unitaries $\Uop_\alpha$.

\begin{prop}
For $\alpha\in\C$, with $\alpha\ne0$, the point spectrum of $\Uop_\alpha:\Fspace\to\Fspace$ 
is empty. Expressed differently, if $\Uop_\alpha f=\lambda f$ holds for some $f\in\Fspace$ and
some constant $\lambda\in\C$, then $f=0$.
\label{prop:emptyspectrum}
\end{prop}

\begin{proof}
If $\Uop_\alpha f=\lambda f$ holds, the isometric property of $\Uop_\alpha$ entails that 
\[
\|f\|_{\Fspace}=\|\Uop_\alpha f\|_{\Fspace}=\|\lambda f\|_{\Fspace}=|\lambda|\,\|f\|_{\Fspace},
\]
and hence either $|\lambda|=1$ or $f=0$, or possibly both. So, unless $|\lambda|=1$, the assertion
of the proposition follows. In the instance that $|\lambda|=1$, we iterate the relation 
$\Uop_\alpha=\lambda f$, and obtain that
\[
\Uop_{n\alpha}f=\Uop_\alpha^ n f=\lambda^ n f,\qquad n=0,1,2,\ldots.
\]
However, in view of \eqref{eq:Focktransmod},
\begin{equation}
\e^{-\frac12|z|^2}|\Uop_{n\alpha}f(z)|=\e^{-\frac12|z-n\alpha|^2}|f(z-n\alpha)|
\label{eq:Focktransmod2}
\end{equation}
while
\begin{equation}
\e^{-\frac12|z|^2}|\lambda^n f(z)|=\e^{-\frac12|z|^2}|f(z)|.
\label{eq:Focktransmod3}
\end{equation}
For fixed $z\in\C$, $|z-n\alpha|\to+\infty$ as $n\to+\infty$, so that in view of Lemma \ref{lem:basic},
the right-hand side of \eqref{eq:Focktransmod2} converges to $0$. Taking \eqref{eq:Focktransmod3}
into consideration we find that $f(z)=0$ is the only possibility.
\end{proof}

We turn to the proof of Theorem \ref{thm:evenodd}. 

\begin{proof}[Proof of Theorem \ref{thm:evenodd}]
Since the assertion is trivial in case $\beta=0$, we assume that $\beta\ne0$ holds instead.
We consider first the case when $\calE=\{0,2,4,\ldots\}$. Then the condition that
\[
f^{(j)}(0)=0,\qquad j\in\calE,
\]
means that $f(-z)=-f(z)$, that is, $f$ is an even function. Likewise, the condition that
\[
(\Uop_\beta f)^{(j)}(0)=0, \qquad j\in\calE^c,
\]
means that $\Uop_\beta f(-z)=\Uop_\lambda f(z)$, so that $\Uop_\beta f$ is an even function.
Written out, the latter condition means that
\[
\Uop_\beta f(z)=\e^{-\frac12|\beta|^2+\bar\beta z}f(z-\beta)=
\e^{-\frac12|\beta|^2-\bar\beta z}f(-z-\beta)=\Uop_\beta f(-z),
\]
which amounts to having
\begin{equation}
f(z-\beta)=\e^{-2\bar\beta z}f(-z-\beta).
\label{eq:symm1}
\end{equation}
We now use the assumption that $f$ is an odd function. Then $f(-z-\beta)=-f(z+\beta)$,
and we obtain from \eqref{eq:symm1} that
\begin{equation}
f(z-\beta)=\e^{-2\bar\beta z}f(-z-\beta)=-\e^{-2\bar\beta z}f(z+\beta).
\label{eq:symm2}
\end{equation}

After a change-of-variables, this is the same as having
\[
f(z)=-\e^{-2|\beta|^2-2\bar\beta z}f(z+2\beta)=-\Uop_{-2\beta}f(z).
\]
But then $f$ is an eigenfunction of $\Uop_{-2\beta}$ with eigenvalue $-1$, which is impossible
unless $f=0$, by Proposition \ref{prop:emptyspectrum}.

The remaining case when $\calE=\{1,2,3,\ldots\}$ is handled in an analogous fashion.
\end{proof}

\section{The induced norm for the even and odd integers}

We shall study the seminorm on $\Fspace$ given by
\begin{equation}
\label{eq:starnorm}
\|f\|^2_{\calE,\beta}:=\sum_{j\in\calE}\frac{|f^{(j)}(0)|^2}{j!}+
\sum_{j\in\calE^c}\frac{|(\Uop_\beta f)^{(j)}(0)|^2}{j!}
\end{equation}
for the cases when $\calE$ consists of the nonnegative even or odd integers and $\beta\in\C$ with
$\beta\ne0$.

\noindent{\sc Case I: $\calE=\{0,2,4,\ldots\}$.}
We readily find that in the notation \eqref{eq:rigidmotions} associated with the
rigid motions group,
\begin{equation}
\label{eq:starnorm}
\|f\|^2_{\calE,\beta}
=\frac14\|f+\Uop_{(-1,0)}f\|_{\Fspace}^2
+\frac14\|\Uop_{\beta}f-\Uop_{(-1,0)}\Uop_{\beta}f\|_{\Fspace}^2
\end{equation}
with \eqref{eq:rigidmotions}, where we observe that
$\Uop_{(-1,0)}f(z)=f(-z)$ for reflection in the origin, while more generally, 
$\Uop_{(\rho,0)}f(z)=f(\rho z)$, for any complex constant $\rho$ with $|\rho|=1$ (rotations).
We put $g:=\Uop_{(\rho,0)}f$, so that $f=\Uop_{(\bar\rho,0)}g$, and find that
\begin{multline}
\|f\|^2_{\calE,\beta}
=\frac14\|\Uop_{(\bar\rho,0)}g+\Uop_{(-1,0)}\Uop_{(\bar\rho,0)}g\|_{\Fspace}^2
+\frac14\|\Uop_{\beta}\Uop_{(\bar\rho,0)}f-\Uop_{(-1,0)}\Uop_{\beta}\Uop_{(\bar\rho,0)}
f\|_{\Fspace}^2
\\
=\frac14\|\Uop_{(\rho,0)}\Uop_{(\bar\rho,0)}g+\Uop_{(\rho,0)}
\Uop_{(-1,0)}\Uop_{(\bar\rho,0)}g\|_{\Fspace}^2
\\
+\frac14\|\Uop_{(\rho,0)}\Uop_{\beta}\Uop_{(\bar\rho,0)}g-\Uop_{(\rho,0)}
\Uop_{(-1,0)}\Uop_{\beta}\Uop_{(\bar\rho,0)}g\|^2_{\Fspace},
\end{multline}
if we use the fact that $\Uop_{(\rho,\alpha)}$ acts unitarily on $\Fspace$ for 
$\alpha,\rho\in\C$ with $|\rho|=1$. Now, in view of \eqref{eq:commrule}, this gives us that
\begin{equation}
\|f\|^2_{\calE,\beta}=\|g\|^2_{\calE,\bar\rho\beta}
=\frac14\|g+
\Uop_{(-1,0)}g\|_{\Fspace}^2
+\frac14\|\Uop_{\bar\rho\beta}g-
\Uop_{(-1,0)}\Uop_{\bar\rho\beta}g\|^2_{\Fspace}.
\end{equation}
By choosing $\rho=\beta/|\beta|$ we get $\bar\rho\beta=|\beta|>0$, 
and by considering in place of $f$ the rotated function $g$,
we realize that we may take in \eqref{eq:starnorm} the parameter $\beta$ to 
be \emph{positive}. This 
will simplify the presentation below.

The Bargmann-Fock space may be identified with the space $L^2(\R)$ via the
\emph{Bargmann transform}
\begin{equation}
\Bop \varphi(z):=(2\pi)^{-\frac14}\int_\R\e^{-\imag zt+\frac{1}{2}z^2-\frac14t^ 2}\varphi(t)\diff t,
\qquad z\in\C.
\end{equation}
It is an exercise to show that 
\begin{equation}
\|\Bop\varphi\|^2_{\Fspace}=\|\varphi\|_{L^2(\R)}^2=\int_\R|\varphi(t)|^2\diff t,
\label{eq:isometry01}
\end{equation}
so that the identification is an isometric isomorphism. We connect $\varphi\in L^2(\R)$ and 
$f\in\Fspace$ via $f=\Bop\varphi$, and realize that with 
\[
\Mop\varphi(t):=\e^{-\frac14t^2}\varphi(t)
\]
we have that
\[
\e^{-\frac12z^2}\Bop\varphi(z)=(2\pi)^{-\frac14}\Fop\Mop\varphi(z),
\]
where $\Fop$ denotes the Fourier transform
\[
\Fop\psi(z)=\int_{\R}\e^{-\imag zt}\psi(t)\diff t.
\]
It follows that with $\check\varphi(t):=\varphi(-t)$, we have the reflection property
\[
\Uop_{(-1,0)}f(z)=f(-z)=(2\pi)^{-\frac14}\e^{\frac12z^2}
\Fop\Mop\check\varphi(z)=\Bop\check\varphi(z).
\]
Moreover, a calculation reveals that since $\beta>0$, it holds that
\[
\Uop_\beta f(z)=\Uop_\beta\Bop\varphi(z)=\Bop\varphi_\beta(z),\qquad \varphi_\beta(t):=
\e^{\imag\beta t}\varphi(t),
\]
and consequently,
\[
\Uop_{(-1,0)}\Uop_\beta f(z)=\Uop_\beta f(-z)=\Bop\varphi_\beta(-z)=\Bop\check\varphi_\beta(z).
\]
This leads to an alternative expression for the norm in \eqref{eq:starnorm}:
\begin{equation}
\|f\|_{\calE,\beta}^2=\frac14\|\varphi+\check\varphi\|^2_{L^2(\R)}+
\frac14\|\varphi_\beta-\check\varphi_\beta\|^2_{L^2(\R)}.
\label{eq:starnorm02}
\end{equation}
To analyze this norm, let us write
\[
\xi:=\varphi+\check\varphi,\qquad \eta:=\varphi_\beta-\check\varphi_\beta,
\]
so that $\xi$ is even while $\eta$ is odd.
We may solve for $\varphi$ in terms of $\xi$ and $\eta$:
\begin{equation}
\varphi(t)=\frac{\eta(t)+\e^{-\imag\beta t}\xi(t)}{2\cos\beta t}.
\label{eq:solvexieta}
\end{equation}
This gives us an alternative approach to the uniqueness theorem 
(Theorem \ref{thm:evenodd}). Indeed, under the assumptions of Theorem \ref{thm:evenodd},
it is given that $\|f\|_{\calE,\beta}=0$,  so by \eqref{eq:starnorm02}, 
$\xi=\varphi+\check\varphi=0$
and $\eta=\varphi_\beta-\check\varphi_\beta=0$. Since $\cos\beta t=0$ holds only for a discrete collection 
of real points $t$, the above identity \eqref{eq:solvexieta}
gives that $\varphi(t)=0$ holds almost everywhere, and hence vanishes as an element of 
$L^2(\R)$. The assertion that $f=0$ now follows by taking the Bargmann-Fock transform.
Moreover, if $\omega_\beta(t):=\cos\beta t$, then by the triangle inequality, 
\[
4\|\omega_\beta\varphi\|^2_{L^2(\R)}\le(\|\eta\|_{L^2(\R)}+\|\xi\|_{L^2(\R)})^2
\le2(\|\eta\|^2_{L^2(\R)}+\|\xi\|^2_{L^2(\R)})=8\|f\|^2_{\calE,\beta}.
\]
At the same time, since 
\[
2\varphi(t)\cos\beta t=\varphi(t)\,\e^{\imag\beta t}+\varphi(t)\,\e^{-\imag\beta t}
=\varphi_\beta(t)+\varphi_{-\beta}(t)
\]
and
\[
\Bop(\varphi_\beta+\varphi_{-\beta})(z)=(\Uop_\beta+\Uop_{-\beta})\Bop\varphi(z)
=\e^{-\frac12|\beta|^2}\big(\e^{\bar\beta z}f(z-\beta)+\e^{-\bar\beta z}f(z+\beta)\big),
\]
where we used that $f=\Bop\varphi$, it follows that
\[
\|(\Uop_\beta+\Uop_{-\beta})f\|^2_{\Fspace}\le8\|f\|_{\calE,\beta}^2.
\]

\noindent{\sc Case II: $\calE=\{1,3,5,\ldots\}$.}
The formula for the induced norm gets slightly modified:
\begin{equation}
\label{eq:starnorm2}
\|f\|^2_{\calE,\beta}
=\frac14\|f-\Uop_{(-1,0)}f\|_{\Fspace}^2
+\frac14\|\Uop_{\beta}f+\Uop_{(-1,0)}\Uop_{\beta}f\|_{\Fspace}^2
\end{equation}
The rotation argument worked out
in the previous case applies here as well, and we may take $\beta>0$. If we appeal to 
the Bargmann transform as in Case I, we express the norm as
\begin{equation}
\|f\|_{\calE,\beta}^2=\frac14\|\varphi-\check\varphi\|^2_{L^2(\R)}+
\frac14\|\varphi_\beta+\check\varphi_\beta\|^2_{L^2(\R)}.
\label{eq:starnorm02.1}
\end{equation}
so that with 
\[
\xi^\circledast:=\varphi-\check\varphi,\qquad \eta^\circledast:=\varphi_\beta+\check\varphi_\beta,
\]
the roles get reversed: $\xi^\circledast$ is odd while $\eta^\circledast$ is even. 
The recovery formula \eqref{eq:solvexieta} remains unchanged,
\begin{equation}
\varphi(t)=\frac{\eta^\circledast(t)+\e^{-\imag\beta t}\xi^\circledast(t)}{2\cos\beta t}.
\label{eq:solvexieta.1}
\end{equation}

\section{The noninterpolation and nonsampling properties}

We turn to the proof of Theorem \ref{thm:samplint}.

\begin{proof}[Proof of Theorem \ref{thm:samplint}]
Without loss of generality, by the rotation argument in the previous section,
we may take $\beta>0$.

We first consider the instance when $\calE=\{0,2,4,\ldots\}$.
In the interpolation problem, the question is whether the symmetrized functions 
\[
f_{\mathrm{even}}(z):=\frac12(f(z)+f(-z)),\qquad
f_{\mathrm{odd},\beta}(z):=\frac12(\Uop_\beta f(z)-\Uop_\beta f(-z)),
\]
may be prescribed freely among the even and odd functions in $\Fspace$, respectively.
These symmetrized functions are related to the functions $\xi,\eta$ of the previous 
subsection:
\[
f_{\mathrm{even}}(z)=\frac12\Bop(\varphi+\check\varphi)(z)=\frac12\Bop\xi(z),\qquad
f_{\mathrm{odd},\beta}(z)=\frac12\Bop(\varphi_\beta-\check\varphi_\beta)(z)=\frac12\Bop\eta(z),
\]
where we recall that $\xi$ is even while $\eta$ is odd. This means that the interpolation
question asks whether we can always, for given $\xi,\eta$ in $L^2(\R)$ that are even and odd, 
respectively, guarantee that the solution $\varphi$ to \eqref{eq:solvexieta} is in $L^2(\R)$.
If e.g. $\xi,\eta$ are continuous and 
\[
\eta(t)+\e^{\imag\beta t}\xi\ne0
\]
whenever $\cos\beta t=0$, the division will cause a singularities at 
the zeros of $\cos\beta t$ which will make it impossible that $\varphi$ given by
\eqref{eq:solvexieta} is in $L^2(\R)$. This means that the interpolation property fails, as claimed.

We turn to the sampling property. In terms of the formula \eqref{eq:solvexieta}, the sampling
property asks for the inequality
\begin{equation}
\label{eq:sampling1.1}
\epsilon\|\varphi\|^2_{L^2(\R)}\le \|\xi\|^2_{L^2(\R)}+\|\eta\|^2_{L^2(\R)}
\end{equation}
for some fixed positive constant $\epsilon$ and all $\varphi\in L^2(\R)$ given by 
\eqref{eq:solvexieta}. To see that this cannot be true, we pick $\xi=0$ and 
\[
\eta(t)=(1+t^2)^{-1}|\cos\beta t|^{\theta}\mathrm{sgn}(t),
\]
where $\mathrm{sgn}(t)$ denotes the sign of $t$, which equals $1$ for positive $t$, $-1$
for negative $t$, and $0$ for $t=0$. We calculate that for $\theta>0$, 
\[
\int_\R|\xi(t)|^2\diff t+\int_\R|\eta(t)|^2\diff t=\int_\R(1+t^2)^{-2}|\cos\beta t|^{2\theta}\diff t\le
\int_\R(1+t^2)^{-2}\diff t\le3,
\]
while
\[
\varphi(t)=\frac{\eta(t)+\e^{\imag\beta t}\xi(t)}{2\cos\beta t}=(1+t^2)^{-1}
\frac{|\cos\beta t|^{\theta}\mathrm{sgn}(t)}{2\cos\beta t}
\]
has
\[
\int_\R|\varphi(t)|^2\diff t=\frac12\int_{\R}|\cos\beta t|^{2\theta-2}(1+t^2)^{-2}\diff t\to+\infty
\]
as $\theta\to0^+$. It now follows that the sampling inequality
\eqref{eq:sampling1.1} cannot hold for a fixed positive constant $\epsilon$.

The remaining case when $\calE=\{1,3,5,\ldots\}$ is entirely analogous and left to the reader.
\end{proof}

\begin{rem}
Not only do we not have that the sampling property fails, also the weaker property 
\eqref{eq:pteval} fails. This can be seen by considering the point evaluation functionals
$f\mapsto f(w)$ as functionals acting on the function $\varphi\in L^2(\R)$, where
$f=\Bop\varphi$.
\end{rem}

\subsection*{Acknowledgements}
The author acknowledges support from Vetenskapsrådet (VR grant 2020-03733), the Leverhulme
trust (grant VP1-2020-007), and by grant 075-15-2021-602 of the Government of the Russian
Federation for the state support of scientific research, carried out under the supervision of leading
scientists.

\end{document}